\documentclass[12pt,reqno,draft]{amsart} 
\usepackage{amssymb,amscd,url}

\begin{document}

%%%%%%%%%%%%%%%%%%%%%%%%%%%%%%%%%%%%%%%%%%%%%%%%%%%%%%%%%%%%%%%%%%%%%%
%% Title and Author Information

\title[Height Estimates for Rational Maps]
{Height Estimates for Equidimensional Dominant Rational Maps}
\date{July 2009}
\author[Joseph Silverman]{Joseph H. Silverman}
\email{jhs@math.brown.edu}
\address{Mathematics Department, Box 1917
         Brown University, Providence, RI 02912 USA}
\thanks{The author's research supported by NSF DMS-0650017 and
DMS-0854755.}

%%%%%%%%%%%%%%%%%%%%%%%%%%%%%%%%%%%%%%%%%%%%%%%%%%%%%%%%%%%%%%%%%%%%%%

\allowdisplaybreaks

\hyphenation{ca-non-i-cal semi-abel-ian}

%%%%%%%%%%%%%%%%%%%%%%%%%%%%%%%%%%%%%%%%%%%%%%%%%%%%%%%%%%%%%%%%%%%%%%
% Theorem environments

\newtheorem{theorem}{Theorem}
\newtheorem{lemma}[theorem]{Lemma}
\newtheorem{conjecture}[theorem]{Conjecture}
\newtheorem{proposition}[theorem]{Proposition}
\newtheorem{corollary}[theorem]{Corollary}
\newtheorem*{claim}{Claim}

\theoremstyle{definition}
% The * surpresses numbering
\newtheorem*{definition}{Definition}
\newtheorem{example}[theorem]{Example}

\theoremstyle{remark}
\newtheorem{remark}[theorem]{Remark}
\newtheorem{question}[theorem]{Question}
\newtheorem*{acknowledgement}{Acknowledgements}

%%%%%%%%%%%%%%%%%%%%%%%%%%%%%%%%%%%%%%%%%%%%%%%%%%%%%%%%%%%%%%%%%%%%%%

%%%%%%%% Set Up Environment for Notation %%%%%%%%%%%%%%
% This is currently set to allow quite wide items to be defined
\newenvironment{notation}[0]{%
  \begin{list}%
    {}%
    {\setlength{\itemindent}{0pt}
     \setlength{\labelwidth}{4\parindent}
     \setlength{\labelsep}{\parindent}
     \setlength{\leftmargin}{5\parindent}
     \setlength{\itemsep}{0pt}
     }%
   }%
  {\end{list}}

%%%%%%%% Set Up Environment for Parts in Theorems %%%%%%%%%%%%%%
\newenvironment{parts}[0]{%
  \begin{list}{}%
    {\setlength{\itemindent}{0pt}
     \setlength{\labelwidth}{1.5\parindent}
     \setlength{\labelsep}{.5\parindent}
     \setlength{\leftmargin}{2\parindent}
     \setlength{\itemsep}{0pt}
     }%
   }%
  {\end{list}}
% Use \Part{(a)}, instead of \item[(a)], to ensure upright font
\newcommand{\Part}[1]{\item[\upshape#1]}

%%%%%%%%%%%%%%%%%%
% Greek Alphabet %
%%%%%%%%%%%%%%%%%%
\renewcommand{\a}{\alpha}
\renewcommand{\b}{\beta}
\newcommand{\g}{\gamma}
\renewcommand{\d}{\delta}
\newcommand{\e}{\epsilon}
\newcommand{\f}{\varphi}
\newcommand{\bfphi}{{\boldsymbol{\f}}}
\renewcommand{\l}{\lambda}
\renewcommand{\k}{\kappa}
\newcommand{\lhat}{\hat\lambda}
\newcommand{\m}{\mu}
\newcommand{\bfmu}{{\boldsymbol{\mu}}}
\renewcommand{\o}{\omega}
\renewcommand{\r}{\rho}
\newcommand{\rbar}{{\bar\rho}}
\newcommand{\s}{\sigma}
\newcommand{\sbar}{{\bar\sigma}}
\renewcommand{\t}{\tau}
\newcommand{\z}{\zeta}

\newcommand{\D}{\Delta}
\newcommand{\G}{\Gamma}
\newcommand{\F}{\Phi}

%%%%%%%%%%%%%%%%%%%%
% Fraktur Alphabet %
%%%%%%%%%%%%%%%%%%%%
\newcommand{\ga}{{\mathfrak{a}}}
\newcommand{\gb}{{\mathfrak{b}}}
\newcommand{\gn}{{\mathfrak{n}}}
\newcommand{\gp}{{\mathfrak{p}}}
\newcommand{\gP}{{\mathfrak{P}}}
\newcommand{\gq}{{\mathfrak{q}}}

%%%%%%%%%%%%%%%%%%%
% Barred Alphabet %
%%%%%%%%%%%%%%%%%%%
\newcommand{\Abar}{{\bar A}}
\newcommand{\Ebar}{{\bar E}}
\newcommand{\Kbar}{{\bar K}}
\newcommand{\Pbar}{{\bar P}}
\newcommand{\Sbar}{{\bar S}}
\newcommand{\Tbar}{{\bar T}}
\newcommand{\ybar}{{\bar y}}
\newcommand{\phibar}{{\bar\f}}

%%%%%%%%%%%%%%%%%%%%%%%%%
% Calligraphic Alphabet %
%%%%%%%%%%%%%%%%%%%%%%%%%
\newcommand{\Acal}{{\mathcal A}}
\newcommand{\Bcal}{{\mathcal B}}
\newcommand{\Ccal}{{\mathcal C}}
\newcommand{\Dcal}{{\mathcal D}}
\newcommand{\Ecal}{{\mathcal E}}
\newcommand{\Fcal}{{\mathcal F}}
\newcommand{\Gcal}{{\mathcal G}}
\newcommand{\Hcal}{{\mathcal H}}
\newcommand{\Ical}{{\mathcal I}}
\newcommand{\Jcal}{{\mathcal J}}
\newcommand{\Kcal}{{\mathcal K}}
\newcommand{\Lcal}{{\mathcal L}}
\newcommand{\Mcal}{{\mathcal M}}
\newcommand{\Ncal}{{\mathcal N}}
\newcommand{\Ocal}{{\mathcal O}}
\newcommand{\Pcal}{{\mathcal P}}
\newcommand{\Qcal}{{\mathcal Q}}
\newcommand{\Rcal}{{\mathcal R}}
\newcommand{\Scal}{{\mathcal S}}
\newcommand{\Tcal}{{\mathcal T}}
\newcommand{\Ucal}{{\mathcal U}}
\newcommand{\Vcal}{{\mathcal V}}
\newcommand{\Wcal}{{\mathcal W}}
\newcommand{\Xcal}{{\mathcal X}}
\newcommand{\Ycal}{{\mathcal Y}}
\newcommand{\Zcal}{{\mathcal Z}}

%%%%%%%%%%%%%%%%%%%%%%%%%%%%
% Blackboard Bold Alphabet %
%%%%%%%%%%%%%%%%%%%%%%%%%%%%
\renewcommand{\AA}{\mathbb{A}}
\newcommand{\BB}{\mathbb{B}}
\newcommand{\CC}{\mathbb{C}}
\newcommand{\FF}{\mathbb{F}}
\newcommand{\GG}{\mathbb{G}}
\newcommand{\PP}{\mathbb{P}}
\newcommand{\QQ}{\mathbb{Q}}
\newcommand{\RR}{\mathbb{R}}
\newcommand{\ZZ}{\mathbb{Z}}

%%%%%%%%%%%%%%%%%%%%%%%%%%
% Boldface Math Alphabet %
%%%%%%%%%%%%%%%%%%%%%%%%%%
\newcommand{\bfa}{{\mathbf a}}
\newcommand{\bfb}{{\mathbf b}}
\newcommand{\bfc}{{\mathbf c}}
\newcommand{\bfe}{{\mathbf e}}
\newcommand{\bff}{{\mathbf f}}
\newcommand{\bfg}{{\mathbf g}}
\newcommand{\bfp}{{\mathbf p}}
\newcommand{\bfr}{{\mathbf r}}
\newcommand{\bfs}{{\mathbf s}}
\newcommand{\bft}{{\mathbf t}}
\newcommand{\bfu}{{\mathbf u}}
\newcommand{\bfv}{{\mathbf v}}
\newcommand{\bfw}{{\mathbf w}}
\newcommand{\bfx}{{\mathbf x}}
\newcommand{\bfy}{{\mathbf y}}
\newcommand{\bfz}{{\mathbf z}}
\newcommand{\bfA}{{\mathbf A}}
\newcommand{\bfF}{{\mathbf F}}
\newcommand{\bfB}{{\mathbf B}}
\newcommand{\bfD}{{\mathbf D}}
\newcommand{\bfG}{{\mathbf G}}
\newcommand{\bfI}{{\mathbf I}}
\newcommand{\bfM}{{\mathbf M}}
\newcommand{\bfzero}{{\boldsymbol{0}}}

%%%%%%%%%%%%%%%%%%%%%%%%%%%%%%
% Miscellaneous New Commands %
%%%%%%%%%%%%%%%%%%%%%%%%%%%%%%
\newcommand{\Aut}{\operatorname{Aut}}
\newcommand{\Base}{\mathcal{B}} % indeterminacy (base) locus
\newcommand{\CM}{\operatorname{CM}}   % CM(O) = set of t with O_t = O
\newcommand{\codim}{\operatorname{codim}}
\newcommand{\Disc}{\operatorname{Disc}}
\newcommand{\Div}{\operatorname{Div}}
\newcommand{\Dom}{\operatorname{Dom}}
\newcommand{\Ell}{\operatorname{Ell}}   % Ell(O) = set of E with CM by O
\newcommand{\End}{\operatorname{End}}
\newcommand{\Fbar}{{\bar{F}}}
\newcommand{\Gal}{\operatorname{Gal}}
\newcommand{\GL}{\operatorname{GL}}
\newcommand{\Index}{\operatorname{Index}}
\newcommand{\Image}{\operatorname{Image}}
\newcommand{\liftable}{{\textup{liftable}}}
\newcommand{\hhat}{{\hat h}}
\newcommand{\Ker}{{\operatorname{ker}}}
\newcommand{\Lift}{\operatorname{Lift}}
\newcommand{\MOD}[1]{~(\textup{mod}~#1)}
\newcommand{\Mor}{\operatorname{Mor}}
\newcommand{\Norm}{{\operatorname{\mathsf{N}}}}
\newcommand{\notdivide}{\nmid}
\newcommand{\normalsubgroup}{\triangleleft}
\newcommand{\NotDom}{\operatorname{NotDom}}
\newcommand{\NS}{\operatorname{NS}}
\newcommand{\odd}{{\operatorname{odd}}}
\newcommand{\onto}{\twoheadrightarrow}
\newcommand{\ord}{\operatorname{ord}}
\newcommand{\Per}{\operatorname{Per}}
\newcommand{\PrePer}{\operatorname{PrePer}}
\newcommand{\PGL}{\operatorname{PGL}}
\newcommand{\Pic}{\operatorname{Pic}}
\newcommand{\Prob}{\operatorname{Prob}}
\newcommand{\Qbar}{{\bar{\QQ}}}
\newcommand{\rank}{\operatorname{rank}}
\newcommand{\Rat}{\operatorname{Rat}}
\newcommand{\rel}{{\textup{rel}}}  %% relative
\newcommand{\Resultant}{\operatorname{Res}}
\renewcommand{\setminus}{\smallsetminus}
\newcommand{\shat}{{\hat s}}
\newcommand{\sing}{{\textup{sing}}} %% singular locus
\newcommand{\Span}{\operatorname{Span}}
\newcommand{\Spec}{\operatorname{Spec}}
\newcommand{\Support}{\operatorname{Supp}}
\newcommand{\That}{{\hat T}}  %% blow-up of T to make P a morphism. 
                              %% maybe should call it V_P to indicate
                              %% dependence on P.
\newcommand{\tors}{{\textup{tors}}}
\newcommand{\Trace}{\operatorname{Trace}}
\newcommand{\tr}{{\textup{tr}}} % for K/k trace
\newcommand{\UHP}{{\mathfrak{h}}}    % Upper half plane
\newcommand{\<}{\langle}
\renewcommand{\>}{\rangle}

\newcommand{\ds}{\displaystyle}
\newcommand{\longhookrightarrow}{\lhook\joinrel\longrightarrow}
\newcommand{\longonto}{\relbar\joinrel\twoheadrightarrow}

%% autonumbered constants for use in inequalities
\newcount\ccount \ccount=0
\def\cc{\global\advance\ccount by1{c_{\the\ccount}}}

%%%%%%%%%%%%%%%%%%%%%%%%%%%%%%%%%%%%%%%%%%%%%%%%%%%%%%%%%%%%%%%%%%%%%%

\begin{abstract}
Let $\f:W\dashrightarrow V$ be a dominant rational map between 
quasi-projective varieties of the same dimension. We give two proofs that
$h_V\bigl(\f(P)\bigr)\gg h_W(P)$ for all points~$P$ in a nonempty
Zariski open subset of~$W$.  For dominant rational
maps~$\f:\PP^n\dashrightarrow\PP^n$, we give a uniform estimate in which the
implied constant depends only on~$n$ and the degree of~$\f$.  As an
application, we prove a specialization theorem for equidimensional
dominant rational maps to semiabelian varieties, providing a
complement to Habegger's recent theorem on unlikely intersections.
\end{abstract}

%% Non-TeX abstract (for ArXiv)
%% Let F : W --> V be a dominant rational map between quasi-projective
%% varieties of the same dimension. We give two proofs that h_V(F(P)) >>
%% h_W(P) for all points P in a nonempty Zariski open subset of W.  For
%% dominant rational maps F : P^n --> P^n, we give a uniform estimate in
%% which the implied constant depends only on n and the degree of F.  As
%% an application, we prove a specialization theorem for equidimensional
%% dominant rational maps to semiabelian varieties, providing a
%% complement to Habegger's recent theorem on unlikely intersections.

\maketitle

%%%%%%%%%%%%%%%%%%%%%%%%%%%%%%%%%%%%%%%%%%%%%%%%%%%%%%%%%%%%%%%%%%%%%%
\section*{Introduction}

A fundamental property of Weil heights~\cite[B.3.2(b)]{MR1745599} is
functoriality for morphisms $\f:W\to V$ of (normal) projective
varieties:
\begin{equation}
  \label{eqn:hWV1}
  h_{V,D}\bigl(\f(P)\bigr) = h_{W,\f^*D}(P)+O(1).
\end{equation}
Functoriality breaks down quite badly for rational maps, as shown
by simple examples such as
\begin{equation}
  \label{eqn:fP2P2X2Y2XZ}
  \f:\PP^2\dashrightarrow\PP^2,\quad
  \f\bigl([X,Y,Z]\bigr) = [X^2,Y^2,XZ],
\end{equation}
which is a map of degree two having fixed points~$[a,a,b]$.
\par
When~$D$ is ample, a simple triangle inequality argument shows
that even for rational maps, there is an upper bound
\[
  h_{V,D}\bigl(\f(P)\bigr) \le h_{W,\f^*D}(P)+C.
\]

Our first result is a lower bound which, although not as
strong as~\eqref{eqn:hWV1}, is sufficiently nontrivial to have
interesting applications.

\begin{theorem}
\label{thm:htlwrbd}
Let $\f:W\dashrightarrow V$ be a dominant rational map between 
quasi-projective varieties, all defined over~$\Qbar$. Assume further that
$\dim(V)=\dim(W)$. Fix height functions~$h_V$ and~$h_W$ on~$V$
and~$W$, respectively, corresponding to ample divisors.  Then there
are constants~$C_1>0$ and~$C_2$ and a nonempty Zariski open set
$U\subset W$ such that
\begin{equation}
  \label{eqn:hWV2}
  h_{V}\bigl(\f(P)\bigr) \ge C_1h_{W}(P)-C_2
  \quad\text{for all $P\in U(\Qbar)$.}
\end{equation}
The constants~$C_1$ and~$C_2$ depend on~$V$,~$W$~$\f$, and the choice
of~$h_V$ and~$h_W$, but are independent of the point~$P$.
\end{theorem}

We will give two proofs of Theorem~\ref{thm:htlwrbd}, the first a short proof
that relies on a numerical criterion involving nef and
big line bundles, the second a direct proof using only elementary
properties of height functions. 

Theorem~\ref{thm:htlwrbd} says that there is a nonempty open set on
which the ratio $h_V\bigl(\f(P)\bigr)\big/h_W(P)$ is bounded below by
a positive constant.  This prompts the following definition.

\begin{definition}
Let $\f:W\dashrightarrow V$ be a rational map between 
quasi-projective varieties, all defined over~$\Qbar$. 
Fix height functions~$h_V$ and~$h_W$ on~$V$ and~$W$, respectively,
corresponding to ample divisors. The \emph{height expansion
coefficient of~$\f$} (relative to the chosen height functions~$h_V$
and~$h_W$) is the quantity
\[
  \mu(\f) = \sup_{\emptyset\ne U\subset W}
  \liminf_{\substack{P\in U(\Qbar)\\ h_W(P)\to\infty\\}}
  \frac{h_V\bigl(\f(P)\bigr)}{h_W(P)},
\]
where the sup is over all nonempty Zariski open subsets of~$W$. (Note that as
we make~$U$ smaller, the liminf becomes larger, so we may restrict attention
to sets~$U$ such that~$\f$ is defined at every point of~$U$.)
\end{definition}

Theorem~\ref{thm:htlwrbd} is equivalent to the assertion that if~$\f$
is equidimensional and dominant, then~$\mu(\f)>0$.  Our second result
gives a uniform bound for dominant self-maps of projective space.

\begin{theorem}
\label{thm:Pnunifbd}
Let $n\ge1$ and $d\ge1$ be integers. There are constants
$C_i=C_i(d,n)$  with $C_1>0$ such that for all dominant
rational maps \text{$\f:\PP^n\dashrightarrow\PP^n$} defined over~$\Qbar$
there is a nonempty Zariski open set \text{$U_\f\subset\PP^n$} such that
\[
  h\bigl(\f(P)\bigr) \ge C_1h(P)-C_2h(\f)-C_3
  \qquad\text{for all $P\in U_\f(\Qbar)$.}
\]
\textup(N.B. The constants~$C_1$,~$C_2$, and~$C_3$ depend only on~$d$
and~$n$, and are independent of the map~$\f$ and the point~$P$. See
Section~$\ref{section:mapsonPn}$ for the exact definition of
the height~$h(\f)$ of a rational map~$\f$.\textup)
\end{theorem}

Theorem~\ref{thm:Pnunifbd} implies that for dominant
maps~$\f:\PP^n\dashrightarrow\PP^n$ of degree~$d$, the height
expansion coefficient~$\mu(\f)$ is bounded below by a constant that
depends only on~$n$ and~$d$. This prompts the definition
\[
  \overline\mu_d(\PP^n)
  \stackrel{\text{def}}{=}
  \inf_{\substack{\f:\PP^n\dashrightarrow\PP^n\\\text{dominant}\\\deg\f=d\\}}
       \mu(\f).
\]
We note that Theorem~\ref{thm:Pnunifbd} implies that
$\overline\mu_d(\PP^n)>0$.
\par
It would be interesting to know the exact value of
$\overline\mu_d(\PP^n)$. It is clear that $\overline\mu_d(\PP^1)=d$,
since every rational self-map of~$\PP^1$ is a morphism.  In
Section~\ref{section:examples} we give examples of maps on~$\PP^n$ for
which we can compute, or estimate, the value of~$\mu(\f)$.
In particular, we prove that
\[
  \overline\mu_d(\PP^n)\le d^{-(n-1)}
  \qquad\text{for all $n\ge2$ and $d\ge2$.}
\]
We also show that certain automorphisms $\f:V\to V$ of~K3 surfaces
satisfy $\mu(\f^n)\le(2+\sqrt3\,)^{-n}$, so even for automorphisms of
varieties, the height expansion coefficient can be arbitrarily small.
\par
For further properties of~$\mu(\f)$, and for a lower bound
for~$\mu(\f)$ that is more closely tied to the geometry of the
map~$\f$, see~\cite{leethesis,leepaper2}.

\begin{acknowledgement}
The author would like to thank Marc Hindry and David Masser for their
assistance, David Cox for suggesting a method of proving
Proposition~\ref{prop:domdn}, and Chong Gyu Lee for a suggestion
regarding Proposition~\ref{prop:Pnhtratsmall}.
\end{acknowledgement}

%%%%%%%%%%%%%%%%%%%%%%%%%%%%%%%%%%%%%%%%%%%%%%%%%%%%%%%%%%%%%%%%%%%%%%
\section{An algebro-geometric proof of Theorem~\ref{thm:htlwrbd}}

We use the following numerical criterion for bigness due to Siu.

\begin{theorem}
\label{thm:siu}
\textup{(Siu~\cite{Siu93}, \cite[Theorem~2.2.15]{Laz04})}
Let $V$ be a projective variety of dimension~$n$, and let~$D$ and~$E$
be nef divisors on~$V$. Assume that
\[
  (D^n) > n(D^{n-1}\cdot E).
\]
Then $D-E$ is big.
\end{theorem}

We recall that ample divisors are nef, that the nef property is
preserved under pull-back by morphisms, and that a divisor is big if
some multiple defines a rational embedding into projective space.
See~\cite[\S\S1.4,2.2]{Laz04} for basic definitions and 
properties of nef and big divisors.

\begin{proof}[Proof of Theorem~$\ref{thm:htlwrbd}$]
Without loss of generality, we may replace $V$ and $W$ by normal
projective varieties, since the statement of the theorem applies
only to points in some Zariski open subset of~$W$.
\par
The map $\f:W\dashrightarrow V$ is only assumed to be rational, so we
resolve the indeterminacy by finding a projective variety~$X$ and
morphisms $\psi:X\to V$ and $\pi:X\to W$ such that~$\pi$ is a
birational morphism and the following diagram
commutes~\cite[II.7.17.3]{MR0463157}:
%%%%%%%%%%%%%%%%%%%%%%%%%%%%%%%%%%%%%%%%%%%%%%%%%%%%%%%%%%%%%%%%%%%%%%
\begin{center}
\begin{picture}(100,95)(0,-10)
\put(0,0){\makebox(0,0)[c]{$W$}}
\multiput(10,0)(10,0){6}{\line(1,0){5}\hspace{5pt}}
\put(70,0){\vector(1,0){10}}
\put(90,0){\makebox(0,0)[c]{$V$}}
\put(0,60){\vector(0,-1){50}}
\put(0,70){\makebox(0,0)[c]{$X$}}
\put(10,60){\vector(4,-3){70}}
\put(7,32){\makebox(0,0)[t]{$\scriptstyle\pi$}}
\put(42,7){\makebox(0,0)[c]{$\scriptstyle\f$}}
\put(48,43){\makebox(0,0)[c]{$\scriptstyle\psi$}}
\end{picture}
\end{center}
%%%%%%%%%%%%%%%%%%%%%%%%%%%%%%%%%%%%%%%%%%%%%%%%%%%%%%%%%%%%%%%%%%%%%%
Let $n=\dim(V)=\dim(W)$. We let~$D\in\Div(W)$ and~$E\in\Div(V)$ be
ample divisors associated to the Weil height functions~$h_W=h_{W,D}$
and~$h_V=h_{V,E}$. The fact that~$D$ is ample implies that its
self-inter\-sec\-tion $(D^n)$ is positive, and then the projection
formula tells us that
\[
 (\pi^*D)^n = (D^n) \ge 1.
\]
Hence we can find an integer~$m\ge1$ satisfying
\[
  m\bigl((\pi^*D)^n\bigr) > n\bigl((\pi^*D)^{n-1}\cdot\psi^*E\bigr).
\]
Multiplying by $m^{n-1}$ yields
\[
  (m\pi^*D)^n > n\bigl((m\pi^*D)^{n-1}\cdot\psi^*E\bigr).
\]
This allows us to apply Siu's theorem (Theorem~\ref{thm:siu}) to the
divisors $m\pi^*D$ and $\psi^*E$ to conclude that $m\pi^*D-\psi^*E$ is
big.  (We are using the facts that ample divisors are nef and that the
pull-back of a nef divisor by a morphism is nef.)  In particular,
there is an integer~$k\ge1$ such that $km\pi^*D-k\psi^*E$ is
effective.  It follows from a standard property of height
functions~\cite[B.3.2(e)]{MR1745599} that there is a nonempty Zariski
open set $U\subset X$ such that
\begin{equation}
  \label{eqn:hXkmDkE1}
  h_{X,km\pi^*D-k\psi^*E}(x) \ge O(1)\quad\text{for all $x\in U(\Qbar)$.}
\end{equation}
(More precisely, we may take~$U$ to be the complement of the base locus
of the divisor~$km\pi^*D-k\psi^*E$.) Functorial properties
of height functions~\cite[B.3.2(b,c)]{MR1745599} tell us that
\begin{align}
  \label{eqn:hXkmDkE2}
  h_{X,km\pi^*D-k\psi^*E}(x)
  &= km h_{X,\pi^*D}(x) - kh_{X,\psi^*E}(x) + O(1) \notag\\
  &= km h_{W,D}\bigl(\pi(x)\bigr) - kh_{V,E}\bigl(\psi(x)\bigr) + O(1).
\end{align}
Combining~\eqref{eqn:hXkmDkE1} and~\eqref{eqn:hXkmDkE2} yields
\[
  h_{W,D}\bigl(\pi(x)\bigr) \ge \frac{1}{m}h_{V,E}\bigl(\psi(x)\bigr) + O(1)
  \quad\text{for all $x\in U(\Qbar)$.}
\]
Using the facts that $\pi$ is surjective, that $\psi=\f\circ\pi$, and
that~$\f$ is defined on an open subset of~$W$, we conclude that there
is an open subset $U'\subset W$ such that
\[
  h_{W,D}(P) \ge \frac{1}{m}h_{V,E}\bigl(\f(P)\bigr) + O(1)
  \quad\text{for all $P\in U'(\Qbar)$.}
\]
This concludes the proof of Theorem~\ref{thm:htlwrbd}.
\end{proof}

%%%%%%%%%%%%%%%%%%%%%%%%%%%%%%%%%%%%%%%%%%%%%%%%%%%%%%%%%%%%%%%%%%%%%%
\section{An elementary height-based proof of Theorem~\ref{thm:htlwrbd}}
\label{section:htproof}

In this section we use basic properties of height functions to give an
alternative proof of Theorem~\ref{thm:htlwrbd}. The key estimate is
the standard inequality relating the height of the roots of a
polynomial to the height of its coefficients. We begin with an
elementary result that will be needed for the proof.

\begin{lemma}
\label{lemma:hWlechV}
Let
\[
  \f : W \dashrightarrow V
\]
be a rational map of projective varieties defined over~$\Qbar$, and
let~$Z_\f$ be the indeterminacy locus of~$\f$, so~$\f$ is well-defined
on~$V\setminus Z_\f$.  Fix height functions~$h_V$ and~$h_W$
corresponding to ample divisors on~$V$ and~$W$, respectively.  Then
there are constants~$C_1>0$ and~$C_2$, such that
\[
  h_V\bigl(\f(P)\bigr) \le C_1 h_W(P) + C_2
  \qquad\text{for all $P\in W(\Kbar)\setminus Z_\f$.}
\]
\end{lemma}
\begin{proof}
This is a standard triangle inequality estimate, but lacking a
suitable reference, we sketch the proof. Replacing~$h_V$ and~$h_W$ by
multiples, we may assume that they correspond to embeddings
$V\subset\PP^n$ and $W\subset\PP^m$.  Extending~$\f$ to a rational
map~$\f:\PP^n\dashrightarrow\PP^m$, this reduces the lemma to the case
that~$V$ and~$W$ are projective spaces, in which
case~\cite[B.2.5(a)]{MR1745599} completes the proof. (More precisely,
we take a finite number of extensions of~$\f$ in order to
cover all of~$V\setminus Z_\f$.)
\end{proof}

\begin{proof}[Proof of  Theorem~$\ref{thm:htlwrbd}$]
The assumptions that~$\dim(V)=\dim(W)$ and that~$\f$ is a dominant
rational map imply that~$\Qbar(W)$ is a finite algebraic extension
of~$\f^*\Qbar(V)$. 
Let~$f\in\Qbar(W)$ be a rational function on~$W$.
Then~$f$ is a root of a polynomial
\[
  X^d + A_1 X^{d-1} + \cdots + A_{d-1} X + A_d
  \qquad\text{with $A_1,\ldots,A_d\in\f^*\Qbar(V)$.}
\]
Let~$U=U_f$ be a nonempty open subset of~$W$ such that~$f$ and all of the 
functions~$A_1,\ldots,A_d$ are defined at every~$P\in U(\Qbar)$.
Then for every~$P\in U(\Qbar)$, the number~$f(P)\in\Qbar$ is a root
of the polynomial
\[
  X^d + A_1(P) X^{d-1} + \cdots + A_{d-1}(P) X + A_d(P).
\]
\par
A standard estimate \cite[VIII.5.9]{MR1329092} relating the heights of the
roots and the coefficients of a polynomial gives
\begin{align}
  \label{eqn:hyt}
  h\bigl(f(P)\bigr) - d\log2 \le h\bigl([1,A_1(P),A_2(P),&\dots,A_d(P)]\bigr)
     \notag\\
  &\text{for all $P\in U(\Qbar)$.}
\end{align}
\par
Since~$A_1,\ldots,A_d\in\f^*\Qbar(V)$, there are
functions~$B_i\in\Qbar(V)$ such that~$A_i=\f^*B_i=B_i\circ\f$.  We
define rational maps
\[
  \a = [1,A_1,\ldots,A_d] : W \dashrightarrow \PP^d,
  \qquad
  \b = [1,B_1,\ldots,B_d] : V \dashrightarrow \PP^d.
\]
With this notation, the estimate~\eqref{eqn:hyt} becomes
\begin{equation}
  \label{eqn:hfPd2haP}
  h\bigl(f(P)\bigr)-d\log2 \le h\bigl(\a(P)\big)
  \quad\text{for all $P\in U(\Qbar)$.}
\end{equation}
\par
Applying the elementary triangle inequality estimate described in
Lemma~\ref{lemma:hWlechV} to the rational map~$\b$ gives
\begin{equation}
  \label{eqn:hbxW}
  h\bigl(\b(Q)\bigr) \le \cc h_V(Q) + \cc
\end{equation}
for all~$Q\in V(\Qbar)$ at which~$B_1,\ldots,B_d$ are defined.  (Here
and in what follows,
the constants~$c_i=c_i(W,V,\f,h_W,h_V,f)>0$ 
are independent of~$P\in U(\Qbar)$.) Applying~\eqref{eqn:hbxW}
with~$Q=\f(P)$ for~$P\in U(\Qbar)$ and combining it
with~\eqref{eqn:hfPd2haP}  yields
\begin{equation}
  \begin{aligned}
    h\bigl(f(P)\bigr)
    &\le \cc h\bigl(\a(P)\bigr) + \cc
      &&\text{from~\eqref{eqn:hfPd2haP},} \\
    &= \cc h\bigl(\b\circ\f(P)\bigr) + \cc
      &&\text{since $\a=\b\circ\f$,} \\
    &\le \cc h_V\bigl(\f(P)\bigr) + \cc
      &&\text{from~\eqref{eqn:hbxW} with $Q=\f(P)$.}
  \end{aligned}
  \label{eqn:hftlehwft}
\end{equation}
\par
The height~$h_W$ is relative to an ample divisor, so taking a multiple
of~$h_W$, we may assume that it is associated to a projective
embedding~$\psi:W\hookrightarrow\PP^n$,
i.e.,~$h_W(P)=h\bigl(\psi(P)\bigr)$. The map~$\psi$ is given by
rational functions, say
\begin{equation}
  \label{eqn:defpsi}
  \psi = [1,f_1,\ldots,f_n]
  \quad\text{with $f_1,\dots,f_n\in\Qbar(W)$}.
\end{equation}
Applying~\eqref{eqn:hftlehwft} to each of~$f_1,\ldots,f_n$, we find that
\[
  \begin{aligned}
    h_W(P)
    &= h\bigl(\psi(P)\bigr) 
      &&\text{from the choice of $\psi$,} \\
    &= h\bigl([1,f_1(P),\dots,f_n(P)]\bigr) 
      &&\text{from \eqref{eqn:defpsi},}\\
    &\le \sum_{i=1}^n h\bigl(f_i(P)\bigr) 
      &&\text{elementary height estimate,} \\
    &\le \cc h_V\bigl(\f(P)\bigr) + \cc
      &&\text{from~\eqref{eqn:hftlehwft}.}
  \end{aligned}
\]
This completes the proof of Theorem~\ref{thm:htlwrbd}.
\end{proof}

\begin{remark}
The open subset~$U$ in Theorem~\ref{thm:htlwrbd} is necessary.  To
see this, consider the map
\[
  \f:\PP^2\dashrightarrow\PP^2,\qquad
  \f=[X^2,Y^2,XZ].
\]
Then~$\f$ is a dominant rational map, but
\[
  \f\bigl([0,\a,\b]\bigr) = [0,1,0]
  \qquad\text{for all $\a\ne0$.}
\]
Thus
\[
  h\bigl(\f\bigl([0,\a,\b]\bigr)\bigr) = h\bigl([0,1,0]\bigr) = 0,
\]
which is certainly not larger than a multiple
of~$h\bigl([0,\a,\b]\bigr)$.  For this example, one can check that
\[
  h\bigl(\f(P)\bigr) \ge h(P)
  \qquad\text{for all $P\in U=\{X\ne0\}$.}
\]
\end{remark}

%%%%%%%%%%%%%%%%%%%%%%%%%%%%%%%%%%%%%%%%%%%%%%%%%%%%%%%%%%%%%%%%%%%%%%
\section{A uniform height estimate for rational self-maps of $\PP^n$}
\label{section:mapsonPn}

We use Theorem~\ref{thm:htlwrbd} to prove Theorem~\ref{thm:Pnunifbd},
which says that there is a uniform lower bound for heights on~$\PP^n$
relative to dominant rational self-maps.
Before starting the proof, we briefly describe the universal parameter
space of dominant rational degree~$d$ self-maps of~$\PP^n$.  To ease
notation, let
\[
  N = \binom{n+d}{n}
\]
be the number of monomials of degree~$d$ in~$n+1$ variables.
A rational map~$\f:\PP^n\to\PP^n$ of degree~$d$ has
the form $\f=[\f_0,\ldots,\f_n]$, where each~$\f_i$ is a homogeneous
polynomial of degree~$d$. Taking the coefficients of~$\f_0,\ldots,\f_n$
as coordinates of a point in projective space, the map~$\f$ corresponds
to a point~$A_\f\in\PP^{(n+1)N-1}$. We define subsets
of~$\PP^{(n+1)N-1}$ as follows:
\begin{align*}
  \Rat_d^n&=\{A_\f : \text{$\f$ has degree $d$} \}, \\
  \Dom_d^n&=\{A_\f : \text{$\f$ is dominant of degree $d$} \}, \\
  \Mor_d^n&=\{A_\f : \text{$\f$ is a morphism} \}.
\end{align*}
We note that
\[
  \Mor_d^n \subset \Dom_d^n \subset \Rat_d^n \subset \PP^{(n+1)N-1}.
\]
\par
We use these inclusions to define the height of a rational
map~$\f\in\Rat_d^n(\Qbar)$ to be the Weil height of the corresponding
point in projective space,
\[
  h(\f) = h_{\PP^{(n+1)N-1}}(A_\f).
\]
Similarly, for a point
\[
  P=(x,A_\f)\in\PP^n_{\Rat_d^n}
  = \PP^n\times\Rat_d^n\subset\PP^n\times\PP^{(n+1)N-1},
\]
we define 
\[
  h(P) = h_{\PP^n}(x)+h(\f)
\]
to be the Weil height relative to the line bundle
$\Ocal_{\PP^n\times\PP^{(n+1)N-1}}(1,1)$.
\par
The set~$\Mor_d^n$ is a nonempty Zariski open subset
of~$\PP^{(n+1)N-1}$; more precisely,~$\Mor_d^n$ is an affine variety,
since it is the complement of a single polynomial, the Macaulay
resultant~\cite{MR1142904}.  Similarly,~$\Rat_d^n$ is a nonempty
Zariski open subset of~$\PP^{(n+1)N-1}$, since~$A_\f\notin\Rat_d^n$ if
and only if there is a non-constant homogeneous polynomial~$\psi$
dividing all of~$\f_0,\ldots,\f_n$. Again, elimination theory says
that for each fixed degree of~$\psi$, the set of such~$A_\f$ is a
Zariski closed set.
\par
The fact that~$\Dom_d^n$ is quasi-projetive is perhaps less clear, and
although undoubtedly well known, for lack of a suitable reference we
sketch the proof.  

\begin{proposition}
\label{prop:domdn}
Over any field of characteristic zero, the parameter space~$\Dom_d^n$
of dominant degree~$d$ rational self-maps of\/~$\PP^n$ is a
\textup(nonempty\textup) Zariski open subset of~$\PP^{(n+1)N-1}$.
\end{proposition}
\begin{proof}
We write~$\f(\PP^n)$ for the image of~$\f$, i.e., $\f(\PP^n)$ is the
Zariski closure in~$\PP^n$ of~$\f(\PP^n\setminus Z_{\f})$,
where~$Z_{\f}$ is the locus of indeterminacy of~$\f$. Then
\[
  A_\f\in\Dom_d^n \quad\Longleftrightarrow\quad\f(\PP^n)=\PP^n.
\]
Equivalently, $A_\f\notin\Dom_d^n$ if and only if~$\f(\PP^n)$ is a
proper Zariski closed subset of~$\PP^n$.
\par
Let~$P\in\PP^N\setminus Z_{\f}$ be a point at which~$\f$ is defined,
and consider the map on the cotangent spaces,
\[
  \f_P^* : \Omega_{\f(\PP^n),\f(P)} \longrightarrow \Omega_{\PP^n,P}.
\]
Our characteristic zero assumption means that we do not have to
worry about inseparability, so we have:
\begin{itemize}
\setlength{\itemsep}{5pt plus5pt}
\item
  If $\dim\f(\PP^n)=n$, then $\f_P^*$ is injective for almost all $P$.
\item
  If $\dim\f(\PP^n)<n$, then $\f_P^*=0$  for all $P$.
\end{itemize}
Hence letting
\[
  J_\f = \det(\partial f_i/\partial x_j)_{0\le i,j\le n}
\]
be the Jacobian determinant, we obtain the characterization
\[
  A_\f \in \Rat_d^n\setminus\Dom_d^n
  \quad\Longleftrightarrow\quad
  J_\f = 0.
\]
The Jacobian~$J_\f$ consists of a certain number of monomials
in~$x_0,\ldots,x_n$ whose coefficients are polynomials in the 
coefficients of~$\f_0,\ldots,\f_n$. The ideal generated by those
coefficients is the ideal that defines the complement
$\Rat_d^n\setminus\Dom_d^n$, which completes the proof that~$\Dom_d^n$
is a quasi-projective subvariety of of~$\Rat_d^n$.
\end{proof}

\begin{proof}[Proof of Theorem~$\ref{thm:Pnunifbd}$]
The idea is to use the universal family over the parameter
space~$\Dom_d^n$ of dominant degree~$d$ rational self-maps of~$\PP^n$.
Directly from the definition of~$\Dom_d^n$, we have a rational map~$\F$
as in the following diagram
\par\noindent
\begin{picture}(200,100)(-120,-20)
  \put(-10,45){\vector(1,-1){40}}
  \put(110,45){\vector(-1,-1){40}}
  \multiput(10,60)(10,0){7}{\line(1,0){5}\hspace{5pt}}
  \put(80,60){\vector(1,0){10}}
  \put(50,0){\makebox(0,0)[ct]{$\ds\Dom_d^n$}}
  \put(0,50){\makebox(0,0)[br]{$\ds\PP^n_{\Dom_d^n}$}}
  \put(100,50){\makebox(0,0)[bl]{$\ds\PP^n_{\Dom_d^n}$}}
  \put(50,65){\makebox(0,0)[bc]{$\scriptstyle\F$}}
  \put(2,25){\makebox(0,0)[r]{$\scriptstyle\pi$}}
  \put(105,25){\makebox(0,0)[r]{$\scriptstyle\pi$}}
\end{picture}
\par
\noindent
with the property that for all~$A_\f\in\Dom_d^n$, the restriction
of~$\F$ to the fiber above~$A_\f$ is the map~$\f$. We first consider
the pull-back of this diagram to a subvariety~$S\subset\Dom_d^n$.

\begin{lemma}
\label{lemma:SinDomdn}
Let $S\subset\Dom_d^n$ be a irreducible subvariety, so we have a
commutative diagram \par\noindent
\begin{picture}(200,100)(-120,-20)
  \put(-5,45){\vector(1,-1){40}}
  \put(100,45){\vector(-1,-1){40}}
  \multiput(10,60)(10,0){7}{\line(1,0){5}\hspace{5pt}}
  \put(80,60){\vector(1,0){10}}
  \put(50,0){\makebox(0,0)[ct]{$\ds S$}}
  \put(-5,55){\makebox(0,0)[br]{$\ds \PP^n_S$}}
  \put(105,55){\makebox(0,0)[bl]{$\ds \PP^n_S$}}
  \put(50,65){\makebox(0,0)[bc]{$\scriptstyle\F$}}
  \put(2,25){\makebox(0,0)[r]{$\scriptstyle\pi$}}
  \put(105,25){\makebox(0,0)[r]{$\scriptstyle\pi$}}
\end{picture}
\par\noindent
Then there are constants $C_1(S)>0$ and $C_2(S)\ge0$ and a nonempty
Zariski open subset~$U_S\subset \PP^n_S$ such that
\[
  h\bigl(\Phi(P)\bigr) \ge C_1(S)h(P) - C_2(S)
  \qquad\text{for all $P\in U_S(\Qbar)$.}
\]
\end{lemma}
\begin{proof}
Apply Theorem~\ref{thm:htlwrbd} to the rational map
$\F:\PP^n_S\dashrightarrow \PP^n_S$.
\end{proof}

We are going to apply Lemma~\ref{lemma:SinDomdn} inductively on the
dimension of~$S$. For a given irreducible $S\subset\Dom_d^n$, we find
constants~$C_i(S)$ and an open set~$U_S\subset\PP^n_S$ as in the
lemma. The complement~$\PP^n_S\setminus U_S$ is a proper Zariski
closed subset, so it is a finite union of irreducible subvarieties,
say~$T_1\cup\cdots\cup T_r$. We separate the~$T_i$'s into two cases,
depending on whether they are horizontal or vertical.  (We say
that~$T\subset\PP^n_S$ is horizontal if~$\pi(T)=S$, and vertical
otherwise.) Let~$T$ be any one of the~$T_i$'s.
\par 
If $T$ is horizontal, then its intersection with every fiber of~$\pi$
is a proper closed subset of the fiber.  So horizontal~$T$ delineate
exceptional sets on each fiber. We let $\Hcal_S$ denote the set
of horizontal~$T_i$.
\par 
If~$T$ is vertical, then~$\pi(T)$ is a proper closed subvariety of~$S$, so
in particular it is a closed subvariety of~$\Dom_d^n$ satisfying
\text{$\dim\bigl(\pi(T)\bigr)<\dim(S)$}. (Note that this is a strict
inequality.)  We let $\Vcal_S$ denote the set of~$\pi(T_i)$ such
that that~$T_i$ is vertical.
\par
We now start the induction with $S=S_0=\Dom_d^n$. This gives us
a set of horizontal subvarieties~$\Hcal_{S_0}$, which we put aside
for later, and a set of $\Vcal_{S_0}$ consisting of proper closed
subvarieties of~$S_0$ associated to vertical subvarieties.
To ease notation, we denote these sets by~$\Hcal_0$ and~$\Vcal_0$.
For each variety~$S\in\Vcal_0$, we apply Lemma~\ref{lemma:SinDomdn} to
obtain sets of horizontal and vertical subvarieties, and we write
\[
  \Hcal_1 = \bigcup_{S\in\Vcal_0} \Hcal_S
  \qquad\text{and}\qquad
  \Vcal_1 = \bigcup_{S\in\Vcal_0} \Vcal_S.
\]
Repeating this process, we inductively obtain two sequences of finite
sets of varieties by the rule
\[
  \Hcal_{k+1} = \bigcup_{S\in\Vcal_k} \Hcal_S
  \qquad\text{and}\qquad
  \Vcal_{k+1} = \bigcup_{S\in\Vcal_k} \Vcal_S
  \qquad\text{for $k=0,1,2,\ldots\,$.}
\]
\par
By construction, the dimension of the varieties in~$\Vcal_k$ are strictly
decreasing as~$k$ increases, so there is a~$K$ such
that~$\Vcal_k=\emptyset$ for $k>K$. More precisely, we can
take $K=\dim(\Dom_d^n)=(n+1)N-1$. We now let
\[
  \overline\Hcal = \bigcup_{k=0}^K \Hcal_k
  \qquad\text{and}\qquad
  \overline\Vcal = \bigcup_{k=0}^K \Vcal_k.
\]
Associated to each~$S\in\overline\Vcal$ are constants~$C_1(S)>0$
and~$C_2(S)\ge0$, and we set
\[
  \overline C_1 = \min_{S\in\overline\Vcal} C_1(S)
  \qquad\text{and}\qquad
  \overline C_2 = \max_{S\in\overline\Vcal} C_2(S).
\]
We also let
\[
  W = \bigcup_{T\in\overline\Hcal}T \subset \PP^n_{\Dom_d^n}.
\]
Note that~$W$ is a proper algebraic subset of $\PP^n_{\Dom_d^n}$
with the property that~$W$ contains no entire fibers of
the fibration $\pi:\PP^n_{\Dom_d^n}\to\Dom_d^n$.
\par
By construction and from the inequality in Lemma~\ref{lemma:SinDomdn},
we have
\begin{equation}
  \label{eqn:hPhiP}
  h\bigl(\Phi(P)\bigr) \ge \overline C_1h(P) - \overline C_2
  \qquad\text{for all $P\in\PP^n_{\Dom_d^n}\setminus W$.}
\end{equation}
We now observe that a point $P\in\PP^n_{\Rat_d^n}$ is really a pair
$P=(x,A_\f)$, and the map~$\F$ is given by $\F(P)=\bigl(\f(x),A_\f\bigr)$.
Further, as noted earlier, the height of~$P=(x,A_\f)$ is simply the
sum $h(P)=h(x)+h(\f)$. Hence~\eqref{eqn:hPhiP} may be rewritten as
\[
  h\bigl(\f(x)\bigr) + h(\f) \ge \overline C_1\bigl(h(x)+h(\f)\bigr)
    - \overline C_2
  \quad\text{for all $(x,A_\f)\in\PP^n_{\Dom_d^n}\setminus W$.}
\]
\par
For any point $A_\f\in\Dom_d^n$, let
\[
  W_\f = \pi^{-1}(A_\f) \cap W \subset \pi^{-1}(A_\f) \cong \PP^n.
\]
By construction, the set~$W_\f$ is a proper Zariski closed subset
of~$\PP^n$.  We have proven that
\[
  h\bigl(\f(x)\bigr)  \ge \overline C_1h(x) - h(\f)
    - \overline C_2 
  \quad\text{for all $A_\f\in\Dom_d^n$ and $x\in\PP^n\setminus W_\f$.}
\]
The constants $\overline C_1$ and $\overline C_2$ depend only on~$d$
and~$n$, which completes the proof of Theorem~\ref{thm:Pnunifbd}.
\end{proof}

%%%%%%%%%%%%%%%%%%%%%%%%%%%%%%%%%%%%%%%%%%%%%%%%%%%%%%%%%%%%%%%%%%%%%%
\section{Height expansion coefficients for $\PP^n$}
\label{section:examples}

We recall that the \emph{height expansion coefficient}
of a rational map $\f:V\to W$ is defined to be the quantity
\[
  \mu(\f) = \sup_{\emptyset\ne U\subset W}
  \liminf_{\substack{P\in U(\Qbar)\\ h_W(P)\to\infty\\}}
  \frac{h_V\bigl(\f(P)\bigr)}{h_W(P)}.
\]
The value of~$\mu(\f)$ clearly depends on the choice of height
functions~$h_V$ and~$h_W$, which in turn depend on the choice of ample
divisors~$D\in\Div(V)$ and~$E\in\Div(W)$. More precisely, it follows
from~\cite[B.3.2(f)]{MR1745599} that~$\mu(\f;D,E)$ depends only on the
algebraic equivalence classes of~$D$ and~$E$.

In the special case that~$W=V$, which is of
interest for example in dynamics, it is natural to take~$D=E$, or
equivalently~$h_W=h_V$. If further~$\NS(V)$ has rank one, as happens
for example when $V=W=\PP^n$, then~$\mu(\f)$ is defined independent of
the choice of~$h_W=h_V$. 

In this section we investigate the height expansion coefficient for dominant
rational self-maps of~$\PP^n$.  For example, if $\f:\PP^n\to\PP^n$ is
a nonconstant \emph{morphism}, then
\[
  h\bigl(\f(P)\bigr)=(\deg\f)h(P)+O(1),
\]
so $\mu(\f)=\deg\f$.  On the other hand, the degree two rational map
$\f:\PP^2\dashrightarrow\PP^2$ described in the
introduction~\eqref{eqn:fP2P2X2Y2XZ} satisfies~$\mu(\f)=1$.
%% *** Check this last statement. Certainly $\mu(\f)\le1$.

We recall the definition
\[
  \overline\mu_d(\PP^n)
  \stackrel{\text{def}}{=}
  \inf_{\substack{\f:\PP^n\dashrightarrow\PP^n\\\text{dominant}\\\deg\f=d\\}}
       \mu(\f).
\]
Theorem~\ref{thm:Pnunifbd} tells us that $\overline\mu_d(\PP^n)>0$.
For $n=1$, every rational map $\PP^1\to\PP^1$ is a morphism, so
$\overline\mu_d(\PP^1)=d$.  The value of $\overline\mu_d(\PP^n)$ for
$n\ge2$ is less clear.  We are going to prove that
\[
  \overline\mu_d(\PP^n) \le \frac{1}{d^{n-1}}.
\]

\begin{question}
If $\f:\PP^n\dashrightarrow\PP^n$ is a rational map that is \emph{not}
a morphism, is it true that $\mu(\f)<\deg\f$ \textup(strict
inequality\textup)?
\end{question}

Let $\psi:\PP^n\dashrightarrow\PP^n$ be a rational map. We write
$Z_{\psi}$ for the locus of indeterminacy of~$\psi$. We also recall the
elementary (triangle inequality) height
estimate~\cite[B.2.5(a)]{MR1745599},
\begin{equation}
  \label{eqn:trianht}
  h\bigl(\psi(P)\bigr) \le (\deg\psi)h(P)+O(1)
  \qquad\text{for all $P\in\PP^n(\Qbar)\setminus Z_{\psi}$.}
\end{equation}
A birational map $\f:\PP^n\dashrightarrow\PP^n$ is called a
\emph{regular affine automorphism} if it is not a morphism, if it
restricts to an automorphism~$\AA^n\to\AA^n$, and if $Z_{\f}\cap
Z_{\f^{-1}}=\emptyset$.  

\begin{proposition}
\label{prop:Pnhtratsmall}
Let $n\ge2$ and
let $\f:\PP^n\dashrightarrow\PP^n$ be a regular affine automorphism.
Then
\[
  \mu(\f) = \frac{1}{(\deg\f)^{\frac{n}{1+\dim Z_{\f}}-1}}.
\]
In particular, there exist regular affine automorphisms of~$\PP^n$ of
every degree $d\ge2$ satisfying $\mu(\f)=d^{-(n-1)}$. Hence
\[
  \overline\mu_d(\PP^n) \le \frac{1}{d^{n-1}}.
\]
\end{proposition}
\begin{proof}
To ease notation, we let
\[
  d_1=\deg(\f),\qquad
  d_2=\deg(\f^{-1}).
\]
We make use of Kawaguchi's theory of canonical heights for regular
affine automorphisms; see~\cite{arxiv0405007}
or~\cite[Exercises~7.17--7.22]{MR2316407}. Kawaguchi constructed
canonical heights under the assumption that~$\f$ satisfies the
following height estimate:
\begin{multline}
  \label{eqn:raaineq}
  \smash[b]{
    \frac{h\bigl(\f(P)\bigr)}{d_1}
    +\frac{h\bigl(\f^{-1}(P)\bigr)}{d_2}
    \ge \left(1+\frac{1}{d_1d_2}\right)h(P)+O(1)
  }
  \\
  \text{for all $P\in\AA^n(\Qbar)$.}
\end{multline}
This estimate was proven for~$n=2$ by Kawaguchi~\cite{arxiv0405007}
and by Chong Gyu Lee~\cite{leepaper2} in general.
\par
Thus there are canonical height functions~$\hhat^+$ and~$\hhat^-$ such
that for all~$P\in\AA^n(\Qbar)$,
\begin{align*}
  \hhat^+(\f P)&=d_1\hhat^+(P),
  &\hhat^{\pm}(P)&\le h(P)+O(1),\\
  \hhat^-(\f^{-1} P)&=d_2\hhat^-(P),
  &h(P) &\le \hhat^+(P)+\hhat^-(P)+O(1).
\end{align*}
We now fix a point~$P\in\AA^n(\Qbar)$ having Zariski dense orbit under
iteration of~$\f$. (It is not hard to see that this is true for most points.)
For each~$k\ge0$ we let $Q_k=\f^{-k}P$ and we compute
\begin{align*}
  h(\f Q_k) &\le \hhat^+(\f Q_k)+\hhat^-(\f Q_k)+O(1) \\
  &= d_1\hhat^+(Q_k) + d_2^{-1}\hhat^-(Q_k) + O(1) 
    &&\text{from properties of $\hhat^\pm$,}\\
  &= d_1\hhat^+(\f^{-k}P) + d_2^{-1}\hhat^-(Q_k) + O(1) 
    &&\text{since $Q_k=\f^{-k}P$,}\\
  &= d_1^{1-k}\hhat^+(P) + d_2^{-1}\hhat^-(Q_k) + O(1) 
    &&\text{since $\hhat^+\circ\f^{-1}=d_1^{-1}\hhat^+$,} \\
  &\le d_1^{1-k}\hhat^+(P) + d_2^{-1}h(Q_k) + O(1)
    &&\text{since $\hhat^-\le h+O(1)$.}
\end{align*}
Hence
\[
  \mu(\f)
  \le \lim_{k\to\infty} \frac{h(\f Q_k)}{h(Q_k)}
  \le \lim_{k\to\infty} 
      \left\{\frac{d_1^{1-k}\hhat^+(P)+O(1)}{h(Q_k)} + d_2^{-1}\right\} 
  = \frac{1}{d_2}.
\]
\par
For the other direction, we compute
\begin{align*}
  \left(1+\frac{1}{d_1d_2}\right)h(P)
  &\le \frac{h\bigl(\f(P)\bigr)}{d_1}
    + \frac{h\bigl(\f^{-1}(P)\bigr)}{d_2} + O(1)
  &&\text{from~\eqref{eqn:raaineq},} \\
  &\le \frac{h\bigl(\f(P)\bigr)}{d_1} + h(P) + O(1)
  &&\text{from \eqref{eqn:trianht}.}
\end{align*}
Now a little bit of algebra yields
\[
  \frac{1}{d_2} \le  \frac{h\bigl(\f(P)\bigr)}{h(P)}
    + O\left(\frac{1}{h(P)}\right).
\]
This holds for all $P\in\AA^n(\Qbar)$, so
\[
  \frac{1}{d_2} \le
  \sup_{\emptyset\ne U\subset \PP^n}
  \liminf_{\substack{P\in U(\Qbar)\\ h(P)\to\infty\\}}
  \frac{h\bigl(\f(P)\bigr)}{h(P)} = \mu(\f).
\]
\par
This completes the proof that $\mu(\f)=1/d_2$. In order to express
this bound in terms of the degree of~$\f$, we let
\[
  \ell_1=1+\dim Z_{\f},\qquad
  \ell_2=1+\dim Z_{\f^{-1}},
\]
and use the relations~\cite[Proposition~2.3.2]{MR1760844}
\[
  \ell_1+\ell_2=n\qquad\text{and}\qquad
  d_2^{\ell_1}=d_1^{\ell_2}.
\]
Thus
\[
  d_2 = d_1^{\ell_2/\ell_1} = d_1^{n/\ell_1-1},
\]
so $\mu(\f)=d_2^{-1}=d_1^{-(n/\ell_1-1)}$.
\par
Finally, for any $n\ge2$ and $d\ge2$ it is easy to write down regular
affine automorphisms of~$\PP^n$ of degree~$d$ 
for which~$Z_{\f}$ has dimension zero,
and for such maps we have~$\mu(\f)=d^{-(n-1)}$.
\end{proof}

We next compute the height expansion coefficient of a rational map
that is not an automorphism.

\begin{proposition}
\label{prop:rhoxinv}
The dominant rational map $\f:\PP^n\dashrightarrow\PP^n$ 
defined by
\[
  \f\bigl([X_0,\ldots,X_n]\bigr) = [X_0^{-1},\ldots,X_n^{-1}]
\]
has height expansion ratio
\[
  \mu(\f) = \frac{1}{n} = \frac{1}{\deg\f}.
\]
\end{proposition}
\begin{proof}
From the definition, it is clear that~$\f$ is dominant and satisfies
\[
  \f\bigl(\f(X)\bigr) = X.
\]
Fix $\e>0$, let~$T$ be a large number, and choose integers
$a_0,\ldots,a_n\in\ZZ$ satisfying
\[
  T^{1-\e} < a_i < T~\text{for all~$i$ \quad and}\quad
  \gcd(a_0,\ldots,a_n)=1.
\]
Consider the point
\[
  P = \f\bigl([a_0,\ldots,a_n]\bigr)
\]
and it's image~$\f(P)=[a_0,\ldots,a_n]$. (We are using the fact
that~$\f\circ\f$ is the identity map.) The height of~$\f(P)$ is
given by
\begin{equation}
  \label{eqn:Hle}
  H\bigl(\f(P)\bigr)
  = H\bigl([a_0,\ldots,a_n]\bigr)
  = \max_{0\le i\le n} |a_i|
  \le T.
\end{equation}
Next we observe that the coordinates of
\[
  P = [\ldots,a_0\cdots a_{j-1}a_{j+1}\cdots a_n,\ldots]
\]
are relatively prime integers, so
%% For example, if $p$ divides the first coordinate $a_1\ldots
%% a_n$, then it divides exactly one of the~$a_j$ with~$1\le j\le n$. But
%% then it does not divide the $j$'th coordinate of~$P$, which is the
%% product of the~$a_i$ excluding~$a_j$.
\begin{equation}
  \label{eqn:Hge}
  H(P) = \max_{0\le j\le n} |a_0\cdots a_{j-1}a_{j+1}\cdots a_n|
  \ge T^{(1-\e)n}.
\end{equation}
Combining~\eqref{eqn:Hle} and~\eqref{eqn:Hge} and taking logarithms
yields
\[
  h\bigl(\f(P)\bigr) \le \frac{1}{(1-\e)n} h(P).
\]
The set of points for which this is valid is Zariski dense, so
\[
  \mu(\f) \le \frac{h\bigl(\f(P)\bigr)}{h(P)} \le  \frac{1}{(1-\e)n}.
\]
This holds for every~$\e>0$, which gives the upper bound $\mu(\f)\le 1/n$.
\par
To prove a lower bound for~$\mu(\f)$, we note that for any rational
map $\psi:\PP^n\dashrightarrow\PP^n$, the triangle inequality 
gives an elementary upper bound
\[
  h\bigl(\psi(P)\bigr) \le (\deg\psi)h(P) + O(1).
\]
Our map~$\f$ has degree~$n$, so 
$h\bigl(\f(P)\bigr) \le nh(P) + O(1)$. Replacing~$P$ with~$\f(P)$ and
using the fact that~$\f^2(P)=P$ yields
\[
  h(P) \le n h\bigl(\f(P)\bigr) + O(1),
\]
so
\[
  \mu(\f) = \liminf_{h(P)\to\infty} \frac{h\bigl(\f(P)\bigr)}{h(P)}
  \ge \liminf_{h(P)\to\infty} 
   \left\{\frac{1}{n} + O\left(\frac{1}{h(P)}\right) \right\}
  = \frac{1}{n}.
\]
This completes the proof that $\mu(\f)=n$.
\end{proof}

\begin{question}
Let $\f:\PP^n\dashrightarrow\PP^n$ be a dominant rational map. Based
on the examples in Propositions~\ref{prop:Pnhtratsmall}
and~\ref{prop:rhoxinv}, it is natural to ask if $\mu(\f)$ always has the
form $d^\kappa$ for some \emph{rational number}~$\kappa$.
\end{question}

\begin{remark}
If $\f,\psi:V\dashrightarrow V$ are rational self-maps of a
variety~$V$, it is natural to ask if there is a relation
between~$\mu(\f)$,~$\mu(\psi)$, and~\text{$\mu(\f\circ\psi)$}. In particular,
for applications to dynamics it would be interesting to
relate~$\mu(\f^n)$ to~$\mu(\f)$. For example, is
$\mu(\f^2)\le\mu(\f)$?  The map described in
Proposition~\ref{prop:rhoxinv} shows that the answer is no, since that
map satisfies~$\f^2(P)=P$, so $\mu(\f)=1/n$ and $\mu(\f^2)=1$.  For 
further discussion of the relation between~$\mu(\f)$,~$\mu(\psi)$,
and~\text{$\mu(\f\circ\psi)$}, see~\cite{leepaper2}.
\end{remark}

As noted earlier, if $\f:\PP^n\to\PP^n$ is a morphism of degree~$d$,
then $\mu(\f)=d$, so nonconstant \emph{morphisms}~$\PP^n\to\PP^n$
never have small height expansion coefficients. It turns out that
self-morphisms of other types of varieties may have height
expansion coefficients that are arbitrarily small, as in the following
example.

\begin{proposition}
\label{prop:K3}
Let $V\subset\PP^2\times\PP^2$ be a non-singular variety of
type $(2,2)$, so~$V$ is a~K3 surface with noncommuting
involutions~$\iota_1$ and~$\iota_2$, and let
$\f=\iota_1\circ\iota_2:V\to V$ be their composition, so~$\f$ is an
automorphism of~$V$. Further, let~$D_1,D_2\in\Div(V)$ be the
pull-backs to~$V$ of divisors of the form~$H\times\PP^2$
and~$\PP^2\times H$, respectively, and let~$D\in\ZZ D_1+\ZZ D_2\subset
\Div(V)$ be an ample divisor in the linear span of~$D_1$ and~$D_2$.
Then for all~$n\ge1$, the height expansion coefficient of~$\f^n$
relative to the divisor~$D$ equals
\[
  \mu(\f^n) = (2+\sqrt3)^{-2n}.
\]
\end{proposition}
\begin{proof}
For basic properties of the K3 surface~$V$,
see~\cite{silverman:K3heights} or~\cite[\S7.4]{MR2316407}.
To ease notation, we let $\a=2+\sqrt3$, and we define divisors
$E^+,E^-\in\Div(V)\otimes\RR$ by
\[
  E^+ = -D_1+\a D_2\qquad\text{and}\qquad E^-=\a D_1-D_2.
\]
We write the given divisor~$D$ as~$D=aE^++bE^-$ and note
that the ampleness of~$D$ is equivalent to~$a>0$ and~$b>0$;
see~\cite{silverman:K3heights}.
\par
There are canonical heights~$\hhat^+$ and~$\hhat^-$ associated
to~$\f$, with properties similar to the canonical heights on~$\PP^2$
described in the proof of Proposition~\ref{prop:Pnhtratsmall}. More
precisely, as described in~\cite{silverman:K3heights}
and~\cite[\S7.4]{MR2316407}, we have
\begin{align*}
  \hhat^+(\f P)&=\a^2\hhat^+(P),
  \hhat^+(P) &= h_{E^+}(P) + O(1), \\
  \hhat^-(\f P)&=\a^{-2}\hhat^-(P),
  \hhat^-(P) &= h_{E^-}(P) + O(1).
\end{align*}
\par
We fix a point~$P\in V(\Qbar)$ with Zariski dense $\f$-orbit,
and for each~$k\ge0$ we let $Q_k=\f^{-k}(P)$. Then for all $n\ge0$ we have
\begin{align*}
  h_D\bigl(\f^n(Q_k)\bigr)
  &= a\hhat^+\bigl(\f^n(Q_k)\bigr)+b\hhat^-\bigl(\f^n(Q_k)\bigr)+O(1) \\
  &=a\a^{2n-2k}\hhat^+(P)+b\a^{-2n}\hhat^-(Q_k) + O(1).
\end{align*}
Hence
\begin{align*}
  \mu(\f^n)
  &\le \lim_{k\to\infty} \frac{h_D\bigl(\f^n(Q_k)\bigr)}{h_D(Q_k)}
   = b\a^{-2n}\lim_{k\to\infty} \frac{\hhat^-(Q_k)}{h_D(Q_k)}\\
  &= b\a^{-2n}\lim_{k\to\infty} \frac{\a^{2k}\hhat^-(P)}
       {a\a^{-2k}\hhat^+(P)+b\a^{2k}\hhat^-(P)+O(1)} 
  =\a^{-2n}.
\end{align*}
\par
For the other inequality, we note that
\begin{align*}
  h_D\bigl(&\f^n(P)\bigr)-\a^{-2n}h_D(P)\\
  &=\bigl(a\a^{2n}\hhat^+(P)+b\a^{-2n}\hhat^-(P)\bigr)
   - \a^{-2n}\bigl(a\hhat^+(P)+b\hhat^-(P)\bigr) +O(1)\\
  &=a(\a^{2n}-\a^{-2n})\hhat^+(P)+O(1).
\end{align*}
Hence
\begin{align*}
  \mu(\f^n)
  &= \liminf_{h_D(P)\to\infty} \frac{h_D\bigl(\f^n(P)\bigr)}{h_D(P)}\\
  &\ge \a^{-2n} + \liminf_{h_D(P)\to\infty} 
        \frac{a(\a^{2n}-\a^{-2n})\hhat^+(P)+O(1)}{h_D(P)}
  \ge \a^{-2n},
\end{align*}
since $a>0$ and $\a^{2n}-\a^{-2n}\ge0$. 
\end{proof}

\begin{remark}
We observe that Proposition~\ref{prop:K3} provides an example in which
the N\'eron--Severi group~$\NS(V)$ has rank greater than one, but the
height expansion coefficient is independent of the divisor class
associated to the chosen height function. In general, for a given
variety~$V$ and map~$\f:V\to V$, it might be interesting to study the
association
\[
  \NS(V)\otimes\RR\longrightarrow\RR,\qquad
  [D]\longmapsto\mu(\f;h_D).
\]
\end{remark}

%%%%%%%%%%%%%%%%%%%%%%%%%%%%%%%%%%%%%%%%%%%%%%%%%%%%%%%%%%%%%%%%%%%%%%
\section{An application to specialization maps}
\label{section:specialization}

We apply Theorem~\ref{thm:htlwrbd} to prove a specialization result.
Specialization theorems over one-dimensional bases are known for
families of abelian varieties~\cite{MR703488} and for products of
multiplicative groups~\cite{BoMaZa99}.  Harbegger~\cite{Hab2,Hab1} has
recently proven strong results for intersections $X^{\circ a}\cap
G^{[\dim X]}$, where~$G$ is a torus or an abelian variety, $G^{[n]}$
is the set of codimension~$n$ subgroups of~$G$, and~$X^{\circ a}$ is
the nonanomalous part of~$X$, and he has announced a forthcoming work
dealing with the case that~$G$ is a semiabelian variety. An immediate
application of Theorem~\ref{thm:htlwrbd} is a complementary
specialization result for dominant rational maps to semiabelian
varieties.

\begin{corollary}
\label{cor:spec}
Let $G/\Qbar$ be a semiabelian variety, let $W/\Qbar$ be a 
projective variety, and let $\f:W\dashrightarrow G$ be a dominant
rational map. Assume further that $\dim(W)=\dim(G)$.
Then there is a nonempty Zariski open subset $U\subset W$ such that
\[
  \bigl\{P\in U(\Qbar) : \text{$\f(P)$ is a torsion point} \bigr\}
\]
is a set of bounded height.
\end{corollary}
\begin{proof}
We let~$U\subset W$ be as in Theorem~\ref{thm:htlwrbd} for the map~$\f$,
so
\[
  h_G\bigl(\f(P)\bigr) \ge C_1h_W(P)-C_2
  \quad\text{for all $P\in U(\Qbar)$.}
\]
It is well-known that the height of torsion points on tori and on
abelian varieties are bounded, and the same is true more generally for
semiabelian varieties; see for example~\cite[appendix]{MR833644}.
Thus there is a~$C_3$ such that~$h_G(Q)\le C_3$ for all $Q\in
G(\Qbar)_\tors$. Hence 
\begin{align*}
  P\in U(\Qbar)~&\text{and}~\f(P)\in G_\tors\\
  &\Longrightarrow 
  h_W(P) \le C_1^{-1}\bigl(h_G\bigl(\f(P)\bigr)+C_2\bigr)
  \le C_1^{-1}(C_3+C_2).
\end{align*}
%% *** Remark: According to math reviews, the appendix to \cite{MR833644}
%% includes a proof, due to Waldschmidt, that heights of torsion points
%% on semiabelian varieties are bounds. In the main part of the paper,
%% Paula proves that for extensions by~$\GG_a$, the height is not
%% bounded. For $p$-torsion, it may grow as rapidly as $O(\log p)$.
%% (This paper contains the proof for ordinary primes, a followup
%% article~\cite{MR1078215} has the proof for supersingular primes.)
%%    Paper of David and Philippon has a definition of canonical heights on
%% semiabelian varieties~\cite{MR1799094}.
%%    There is a also a paper~\cite{MR1697451} which deals with computing
%% canonical hts using group extensions. Any chance it could be useful
%% for Lehmer problem?
\end{proof}

\begin{remark}
We mention that a version of Corollary~\ref{cor:spec} remains valid
when~$\f$ is not dominant. More precisely, if the image of~$\f$ is not
contained in the translate of a subgroup of~$G$, then $\f(W)\cap
G_\tors$ is not Zariski dense in~$\f(W)$, so~$\f^{-1}(G_\tors)$ is not
Zariski dense in~$W$. This follows immediately from a general version
of the Manin--Mumford conjecture for semiabelian varieties proven by
Hindry~\cite{MR969244}, generalizing the proof for tori by
Laurent~\cite{MR767195} and for abelian varieties by
Raynaud~\cite{MR717600}.
\end{remark}

%%%%%%%%%%%%%%%%%%%%%%%%%%%%%%%%%%%%%%%%%%%%%%%%%%%%%%%%%%%%%%%%%%%%%%%%
% Bibliography
%%%%%%%%%%%%%%%%%%%%%%%%%%%%%%%%%%%%%%%%%%%%%%%%%%%%%%%%%%%%%%%%%%%%%%%%

%% \begin{thebibliography}{99}
%% \itemsep=\smallskipamount
%% \end{thebibliography}

%% \bibliographystyle{plain}
%% \bibliography{Specialization}

\begin{thebibliography}{10}

\bibitem{BoMaZa99}
E.~Bombieri, D.~Masser, and U.~Zannier.
\newblock Intersecting a curve with algebraic subgroups of multiplicative
  groups.
\newblock {\em Internat. Math. Res. Notices}, (20):1119--1140, 1999.

\bibitem{MR833644}
Sarah~P. Cohen.
\newblock Heights of torsion points on commutative group varieties.
\newblock {\em Proc. London Math. Soc. (3)}, 52(3):427--444, 1986.

\bibitem{Hab2}
P.~Habegger.
\newblock Intersecting subvarieties of abelian varieties with algebraic
  subgroups of complementary dimension.
\newblock {\em Invent. Math.}, 176(2):405--447, 2009.

\bibitem{Hab1}
P.~Habegger.
\newblock On the bounded height conjecture.
\newblock {\em Int. Math. Res. Not. IMRN}, (5):860--886, 2009.

\bibitem{MR0463157}
Robin Hartshorne.
\newblock {\em Algebraic geometry}.
\newblock Springer-Verlag, New York, 1977.
\newblock Graduate Texts in Mathematics, No. 52.

\bibitem{MR969244}
Marc Hindry.
\newblock Autour d'une conjecture de {S}erge {L}ang.
\newblock {\em Invent. Math.}, 94(3):575--603, 1988.

\bibitem{MR1745599}
Marc Hindry and Joseph~H. Silverman.
\newblock {\em Diophantine Geometry: An Introduction}, volume 201 of {\em
  Graduate Texts in Mathematics}.
\newblock Springer-Verlag, New York, 2000.

\bibitem{MR1142904}
J.-P. Jouanolou.
\newblock Le formalisme du r\'esultant.
\newblock {\em Adv. Math.}, 90(2):117--263, 1991.

\bibitem{arxiv0405007}
Shu Kawaguchi.
\newblock Canonical height functions for affine plane automorphisms.
\newblock {\em Math. Ann.}, 335(2):285--310, 2006.

\bibitem{MR767195}
Michel Laurent.
\newblock \'{E}quations diophantiennes exponentielles.
\newblock {\em Invent. Math.}, 78(2):299--327, 1984.

\bibitem{Laz04}
Robert Lazarsfeld.
\newblock {\em Positivity in algebraic geometry. {I}}, volume~48 of {\em
  Ergebnisse der Mathematik und ihrer Grenzgebiete. 3. Folge. A Series of
  Modern Surveys in Mathematics}.
\newblock Springer-Verlag, Berlin, 2004.

\bibitem{leepaper2}
Chong~Gyu Lee.
\newblock The maximal ratio of divisor coefficients and an upper bound for the
  height of rational maps.
\newblock preprint, August 2009.

\bibitem{leethesis}
Chong~Gyu Lee.
\newblock {\em Height estimates for rational maps (tentative title)}.
\newblock PhD thesis, Brown University, 2010 (expected).

\bibitem{MR717600}
M.~Raynaud.
\newblock Sous-vari\'et\'es d'une vari\'et\'e ab\'elienne et points de torsion.
\newblock In {\em Arithmetic and Geometry, Vol. I}, volume~35 of {\em Progr.
  Math.}, pages 327--352. Birkh\"auser Boston, Boston, MA, 1983.

\bibitem{MR1760844}
Nessim Sibony.
\newblock Dynamique des applications rationnelles de {$\mathbb{P}\sp k$}.
\newblock In {\em Dynamique et g\'eom\'etrie complexes (Lyon, 1997)}, volume~8
  of {\em Panor. Synth\`eses}, pages ix--x, xi--xii, 97--185. Soc. Math.
  France, Paris, 1999.

\bibitem{MR703488}
Joseph~H. Silverman.
\newblock Heights and the specialization map for families of abelian varieties.
\newblock {\em J. Reine Angew. Math.}, 342:197--211, 1983.

\bibitem{silverman:K3heights}
Joseph~H. Silverman.
\newblock Rational points on {$K3$} surfaces: a new canonical height.
\newblock {\em Invent. Math.}, 105(2):347--373, 1991.

\bibitem{MR1329092}
Joseph~H. Silverman.
\newblock {\em The Arithmetic of Elliptic Curves}, volume 106 of {\em Graduate
  Texts in Mathematics}.
\newblock Springer-Verlag, New York, 1992.

\bibitem{MR2316407}
Joseph~H. Silverman.
\newblock {\em The Arithmetic of Dynamical Systems}, volume 241 of {\em
  Graduate Texts in Mathematics}.
\newblock Springer, New York, 2007.

\bibitem{Siu93}
Yum~Tong Siu.
\newblock An effective {M}atsusaka big theorem.
\newblock {\em Ann. Inst. Fourier (Grenoble)}, 43(5):1387--1405, 1993.

\end{thebibliography}

\end{document}